\documentclass[12pt,a4paper]{article}
\usepackage[latin2]{inputenc}
\usepackage{amsmath,amsthm}
\usepackage{amsfonts}
\usepackage{amssymb}
\usepackage{graphicx}
\usepackage{array}
\usepackage[left=2.00cm, right=2.00cm, top=2.50cm, bottom=2.50cm]{geometry}
\title{Integral Bases of Pure fields}

\DeclareMathOperator{\lcm}{lcm}
\DeclareMathOperator{\ind}{ind}

\def\Q{{\mathbb Q}}
\def\N{{\mathbb N}}
\def\Z{{\mathbb Z}}
\def\R{{\mathbb R}}

\def\F{{\mathbb F}}

\newtheorem{lemma}{Lemma}
\newtheorem{theorem}[lemma]{Theorem}
\newtheorem{corollary}[lemma]{Corollary}
\newtheorem{proposition}[lemma]{Proposition}

\theoremstyle{definition}
\newtheorem{remark}[lemma]{Remark}

\title{
Integral basis of pure fields with square-free parameter
}
\author{
L\'aszl\'o Remete\thanks{
        Research supported by the \'UNKP-18-3 new national excellence program of the ministry of human capacities.} \\
{\small University of Debrecen, Mathematical Institute} \\
{\small H--4002 Debrecen Pf.400., Hungary,}\\ 
{\small e--mail: remete.laszlo@science.unideb.hu}
}
\date{}
\begin{document}
\maketitle
\thispagestyle{empty}

\renewcommand{\thefootnote}{\arabic{footnote}}
\setcounter{footnote}{0}

\begin{abstract}
Let $m\neq0,\pm1$ and $n\geq 2$ be integers. The ring of algebraic integers of the pure fields of type $\Q(\sqrt[n]{m})$ is explicitly known for $n=2,3,4$. It is well known that for $n=2$, an integral basis of the pure quadratic fields can be given parametrically, by using the remainder of the square-free part of $m$ modulo 4. Such characterisation of an integral basis also exists for cubic and quartic pure fields, but for higher degree pure fields there are only results for special cases.\\
In this paper we explicitly give an integral basis of the field $\Q(\sqrt[n]{m})$, where $m\neq\pm1$ is square-free. Furthermore, we show that similarly to the quadratic case, an integral basis of $\Q(\sqrt[n]{m})$ is repeating periodically in $m$ with period length depending on $n$.
\end{abstract}

\renewcommand{\thefootnote}{}

\section{Introduction}

Let $m\neq 0,\pm 1$, and $n\geq 2$ be integers. There is an extensive literature of pure fields of type $K=\Q(\sqrt[n]{m})$.
It is well-known, that if $m$ is square-free, then an integral basis of $K$ is given by
$$\left\lbrace
\begin{array}{lr}
(1,\sqrt{m}),&\mbox{ if } m\equiv2,3\mod{4},\\
\left(1,\frac{1+\sqrt{m}}{2}\right),& \mbox{ if } m\equiv1\mod{4}.
\end{array}\right.$$
For pure cubic fields $K=\Q(\sqrt[3]{m})$, R.Dedekind \cite{ded} explicitly gave an integral basis, which depends on the remainder of $m$ modulo 9. In this case $m$ is assumed to be cube-free,without loss of generality, so we can write uniquely $m=ab^2$, where $a$ and $b$ are square-free. An integral basis of pure cubic fields is given by
$$\left\lbrace
\begin{array}{lr}
\left(1,\sqrt[3]{m},\frac{\sqrt[3]{m}^2}{b}\right),&\mbox{ if } m\equiv0,2,3,4,5,6,7\mod{9},\\
\left(1,\sqrt[3]{m},\frac{b^2+b^2\cdot\sqrt[3]{m}+\sqrt[3]{m}^2}{3b}\right),& \mbox{ if } m\equiv1\mod{9},\\
\left(1,\sqrt[3]{m},\frac{b^2-b^2\cdot\sqrt[3]{m}+\sqrt[3]{m}^2}{3b}\right),& \mbox{ if } m\equiv-1\mod{9}.
\end{array}\right.$$
The pure quartic fields were studied by T.Funakura \cite{fun}, who explicitly gave an integral basis of the fields $\Q(\sqrt[4]{m})$ by using that these fields are quadratic extensions of $\Q(\sqrt{m})$.\\\\
For pure number fields of higher degree, there are only special results. If $n$ is square-free, then the discriminant and an integral basis of pure fields $\Q(\sqrt[n]{m})$ were given in W.E.H.Berwick \cite{ber}, and if $n$ and $m$ are coprime then in K.Okutsu \cite{oku}. A.Hameed and T.Nakahara \cite{n3} constructed an integral basis of pure octic fields $\Q(\sqrt[8]{m})$, where $m$ is square-free. A.Hameed, T.Nakahara, S.M.Husnine and S.Ahmad \cite{nnn} studied the existence of canonical number systems in pure fields, which is equivalent to the existence of a power integral basis in these fields. They proved, that if $m$ is square-free and all prime divisors of $n$ divides $m$, then the ring of integers of $\Q(\sqrt[n]{m})$ is $\Z[\sqrt[n]{m}]$.
T.A.Gassert \cite{maps} proved, that the ring of integers of $\Q(\sqrt[n]{m})$ is $\Z[\sqrt[n]{m}]$, if and only if $m$ is square-free, and for all primes $p\mid n$, $p^2$ does not divide $m^p-m$. \\
I.Gaál and L.Remete \cite{ibm} explicitly gave an integral basis of the pure fields $\Q(\sqrt[n]{m})$, where $2<n<10$, and $m$ is square-free. Further, they showed that similarly to the quadratic case, if $m$ is square-free, then an integral basis of these fields is repeating periodically in $m$ with period length $n_0$, which depends only on $n$. For any degree, they proved that this period length is at most $n_0=n^{\lfloor n^2/2\rfloor}$, and for $2<n<10$ this can be reduced to $n_0=n^2$.\\\\
Let $f_m(X)=X^n-m$, where $m$ is a square-free integer. Consider the algebraic number fields $K_m=\Q(\alpha_m)$, where $\alpha_m$ is a root of $f_m(X)$. These number fields are of degree $n$, since we assumed that $m$ is square-free, and therefore $f_m(X)$ is irreducible. \\
We will say that an integral basis of the fields $K_m$ is \textit{repeating periodically modulo $n_0$}, if for each residue class $r$ modulo $n_0$, there exist polynomials $h^{(r)}_i(X)\in\Q[X]$, $(i=0,\ldots n-1)$, such that, if $m\equiv r\mod{n_0}$, then  
$$\left(h^{(r)}_0(\alpha_m),h^{(r)}_1(\alpha_m),h^{(r)}_2(\alpha_m),\ldots,h^{(r)}_{n-1}(\alpha_m)\right), $$
is an integral basis of $K_m$.

\begin{remark}
For example, in the case of the pure cubic fields, let $n_0=9$, and
$$h^{(r)}_0(X)=1,\quad h^{(r)}_1(X)=X,\quad h^{(r)}_2(X)=X^2,\quad(r=2,3,4,5,6,7),$$
$$h^{(1)}_0(X)=1,\quad h^{(1)}_1(X)=X,\quad h^{(1)}_2(X)=\frac{1+X+X^2}{3},$$
$$h^{(8)}_0(X)=1,\quad h^{(8)}_1(X)=X,\quad h^{(8)}_2(X)=\frac{1-X+X^2}{3}.$$
These polynomials fulfil the conditions in the definition, so we can say that an integral basis of the pure cubic fields $\Q(\sqrt[3]{m})$, where $m$ is square-free, is repeating periodically modulo 9.
\end{remark}

In this approach, the results above imply, that if $m$ is square-free, then 
\begin{itemize}
\item an integral basis of $\Q(\sqrt{m})$ is repeating periodically modulo $4$,
\item an integral basis of $\Q(\sqrt[3]{m})$ is repeating periodically modulo $9$,
\item an integral basis of $\Q(\sqrt[4]{m})$ is repeating periodically modulo $8$,
\item an integral basis of $\Q(\sqrt[n]{m})$ is repeating periodically modulo $n^2$, for $2\leq n\leq10$,
\item an integral basis of $\Q(\sqrt[n]{m})$ is repeating periodically modulo $n^{\lfloor n^2/2\rfloor}$, for any $n$.
\end{itemize}

\begin{remark}
The above periodicity implies, that for any number field $\Q(\sqrt[n]{m})$ of degree $n$, if $m_1$ and $m_2$ are square-free integers belonging to the same residue class modulo $n_0$, then for any $h(X)\in\Q[X]$, $h(\sqrt[n]{m_1})$ is an algebraic integer if and only if $h(\sqrt[n]{m_2})$ is an algebraic integer.
\end{remark}
Similar phenomenon was proved for simplest sextic fields, if $m^2+3m+9$ is square-free (see I.Gaál and L.Remete \cite{sextic}).\\
The aim of this paper is to generalize the results of \cite{ibm}, and to give an integral basis of any pure field $\Q(\sqrt[n]{m})$, where $m\neq\pm1$ is a square-free integer.\\
In Section \ref{s1} we show, that if $m$ is square-free, and $M=\Q(\sqrt[n]{m})$, then 
$$\frac{n}{\gcd(n,m)}\cdot {\cal O}_M\subset\Z[\sqrt[n]{m}],$$
where ${\cal O}_M$ denotes the ring of integers of $M$. This will imply, that if an integral basis of these fields is repeating periodically, then the period length will be independent from $m$.\\
In Section \ref{s2} we give an integral basis of $\Q(\sqrt[p^k]{m})$, where $p$ is a prime, and $m\neq\pm1$ is square-free, by using the theory of Newton-polygons. We obtain that an integral basis of $\Q(\sqrt[p^k]{m})$ is repeating periodically modulo $n_0=p^{k+1}$. By using these results, in Section \ref{ex1} we give an integral basis of the infinite parametric family of number fields $\Q(\sqrt[9]{27t+1})$.\\
In Section \ref{s3}, we extend our results for composite values of $n$. We prove that if $m$ is square-free, and the prime decomposition of $n$ is 
$$n=\prod_{i=1}^{r}p_i^{k_i},$$
then  an integral basis of $\Q(\sqrt[n]{m})$ is repeating periodically modulo $n_0$, where 
$$n_0=\prod_{i=1}^{r}p_i^{k_i+1}.$$
As an example demonstrating our method, in Section \ref{ex2} we give an integral basis of the infinite parametric family of number fields of type $\Q(\sqrt[12]{m})$, where $m$ is a square-free integer.

\section{Preliminaries}\label{s1}
Let $\alpha$ be an algebraic integer of degree $n$, and $M=\Q(\alpha)$. Let ${\cal O}_M$ denote the ring of integers in $M$.
The \textit{index} of $\alpha$ is defined by
$$\ind(\alpha):=({\cal O}_M^+:\Z[\alpha]^+).$$
The index of $\alpha$ is nothing else, than the product of the denominators of the elements of an integral basis of $M$ written in the basis $(1,\alpha,\ldots,\alpha^{n-1})$.
It can be shown (see for eg. \cite{Nark} Chapter II. Proposition 2.13), that if $D(\alpha)$ denotes the discriminant of the minimal polynomial of $\alpha$ (which is equal to the discriminant of the basis $(1,\alpha,\ldots,\alpha^{n-1})$), and $D_M$ denotes the discriminant of the field $M$, then
$$D(\alpha)=\ind(\alpha)^2\cdot D_M.$$
Let $v_p:\Q\longrightarrow\Z$ denote the $p$-adic valuation of a rational number.\\
The \textit{$p$-index} of an algebraic integer is defined by
$$\ind_p(\alpha)=v_p(\ind(\alpha)).$$
By this definition we have
$$\ind(\alpha)=\prod_{p\mid D(\alpha)}p^{\ind_p(\alpha)}.$$
We would like to show that if $m$ is square-free, then for any $n\geq2$ there exists a "small" $n_0\in\N$, such that an integral basis of the fields $\Q(\sqrt[n]{m})$ is repeating periodically modulo $n_0$. I.Gaál and L.Remete \cite{ibm} proved this with $n_0=n^{\lfloor n^2/2\rfloor}$, but it is far from being the best possible period length. \\
First we show that the denominator of any algebraic integer in $\Q(\sqrt[n]{m})$ (written in the basis $(1,\sqrt[n]{m},\ldots,\sqrt[n]{m^{n-1}})$, with simplified coefficients), divides $n$ and is coprime to $m$.

\begin{lemma}\label{ein}
(\cite{Nark} Chapter II., Lemma 2.17)
Let $M$ be a number field of degree $n$, and let $\alpha\in M$ be a nonzero algebraic integer of degree $n$. Suppose that the minimal polynomial of $\alpha$ is Eisenstein with respect to the prime $p$. Then $\ind_p(\alpha)=0$.
\end{lemma}
If $m\neq\pm1$ is square-free, then $X^n-m$ is Eisenstein with respect to all prime divisor of $m$.\\
Hence, if $m\neq\pm1$ is square-free, then $\ind(\sqrt[n]{m})$ is coprime to $m$, and
$$\ind(\sqrt[n]{m})=\prod_{p\mid C} p^{\ind_p(\sqrt[n]{m})},$$
where 
$$C=\frac{D(\sqrt[n]{m})}{\gcd(D(\sqrt[n]{m}),m)}.$$
Furthermore, since $\ind(\sqrt[n]{m})^2$ divides $ D(\sqrt[n]{m})=(-1)^{\frac{n(n-1)}{2}}\cdot n^nm^{n-1}$, hence
$$\ind(\sqrt[n]{m})^2\left|\left(\frac{n}{\gcd(n,m)}\right)^n.\right.$$ 
\begin{corollary}\label{den}
If $m\neq\pm1$ is square-free, then any algebraic integer in  $\Q(\sqrt[n]{m})$ can be written in the basis $(1,\sqrt[n]{m},\ldots,\sqrt[n]{m^{n-1}})$, with rational coefficients having denominator $d$, for which $\gcd(d,m)=1$.
\end{corollary}
Let $M$ be an algebraic number field of degree $n$. If $(\alpha_0,\alpha_1,\ldots,\alpha_{n-1})$ is a $\Q$-basis of $M$, then its dual basis is $(\gamma_0,\gamma_1,\ldots,\gamma_{n-1})$, where $\gamma_i\in M$, such that 
$$Tr_{M/\Q}(\gamma_i\cdot \alpha_j)=\delta_{ij}=\left\lbrace
\begin{array}{ll}
1, & $if $i=j,\\
0, & $otherwise.$
\end{array}
\right.
$$
The dual basis of any basis of $M$ is always uniquely exists (see for eg. \cite{pohst}).\\
Let $(\alpha_0,\alpha_1,\ldots,\alpha_{n-1})$ be a $\Q$-basis of $M$ containing algebraic integers, and $(\gamma_0,\gamma_1,\ldots,\gamma_{n-1})$ its dual basis. Let $\beta\in {\cal O}_M$, and represent it in the dual basis:
$$\beta=f_0\gamma_0+f_1\gamma_1+\ldots+f_{n-1}\gamma_{n-1},\qquad (f_0,\ldots f_{n-1}\in \Q)$$
Let us consider the trace of the product of $\beta$ and $\alpha_i$, $(i=0,1,\ldots,n-1)$:
$$Tr_{M/\Q}(\beta\cdot\alpha_i)=f_{0}Tr_{M/\Q}(\gamma_0\cdot\alpha_i)+f_{1}Tr_{M/\Q}(\gamma_1\cdot\alpha_i)+\ldots+f_{n-1}Tr_{M/\Q}(\gamma_{n-1}\cdot\alpha_i)=f_i.$$
Since $\beta$ and $\alpha_i$ are algebraic integers, hence $Tr_{M/\Q}(\beta\cdot\alpha_i)\in\Z$, whence the coefficients of $\beta$ with respect to the basis $(\gamma_0,\gamma_1,\ldots,\gamma_{n-1})$, are rational integers.

\begin{theorem}\label{nes}
Let $m\neq\pm1$ be a square-free integer and $M=\Q(\sqrt[n]{m})$. Then 
$$\frac{n}{\gcd(n,m)}\cdot {\cal O}_M\subset \Z[\sqrt[n]{m}].$$
\end{theorem}

\begin{proof}
To prove the theorem, it is sufficient to show, that the denominators of the coefficients of any algebraic integer in $M$, with respect to the $\Q$-basis $(1,\sqrt[n]{m},\ldots,\sqrt[n]{m^{n-1}})$, divide $n$ and are coprime to $m$. It is straightforward to check, that the dual basis of $(1,\sqrt[n]{m},\ldots,\sqrt[n]{m^{n-1}})$, is 
$$\left(\frac{1}{n},\frac{\sqrt[n]{m^{n-1}}}{n\cdot m},\frac{\sqrt[n]{m^{n-2}}}{n\cdot m},\ldots,\frac{\sqrt[n]{m}}{n\cdot m}\right).$$
Since it is a dual basis of a basis containing algebraic integers, any algebraic integer in $M$ have rational integer coefficients, with respect to this basis. It means, that we can write any algebraic integer in the basis $(1,\sqrt[n]{m},\ldots,\sqrt[n]{m^{n-1}})$, with rational coefficients having a common denominator dividing $n\cdot m$, but by Corollary \ref{den}, these denominators are coprime to $m$, therefore, they divide $n$.
\end{proof}

\subsection{Newton polygons in algebraic number theory}

Based on \cite{ore},\cite{GMN} and \cite{FMN} we briefly describe here the method of \O.Ore using Newton polygons to determine the $p$-index of an algebraic integer.\\
Extend $v_p$ to a discrete valuation of $\Q(X)$ by letting it act in the following way on polynomials:
$$v_p(a_0+a_1X+\ldots+a_rX^r):=\underset{0\leq i\leq r}\min\left\lbrace v_p(a_i)\right\rbrace.$$
Let $f(X)\in\Z[X]$, and $\phi(X)\in\Z[X]$ be a monic lift of an irreducible factor of $f(X)$ modulo $p$. There is a unique expression of $f(X)$ as a development of $\phi(X)$, that is
$$f(X)=a_0(X)+a_1(x)\cdot \phi(X)+a_2(X)\cdot\phi(X)^2+\ldots+a_r(X)\cdot\phi(X)^r,$$
where $a_i(X)\in\Z[X]$ and $\deg(a_i)<\deg(\phi)$.\\
For $0\leq i\leq r$, let $u_i=v_p(a_i(X))$. The $\phi$-Newton polygon of $f$ is defined to be the lower convex hull of the points
$$\left\lbrace (i,u_i) \in \R^2: 0\leq i\leq r\right\rbrace .$$
The polygon determined by the sides of the $\phi$-Newton polygon with negative slopes is called the {\bf principal $\phi$-Newton polygon of $f$}, and it is denoted by $N$.\\
The \textbf{$\phi$-index} of $f$ is the product of the degree of $\phi$ and the number of points with integral coordinates lying in the first quadrant, on or below $N$. This $\phi$-index is denoted by $\ind_\phi(f)$.\\\\
Let $\overline{f}$ denote the reduction of $f\in\Z[X]\mod{p}$. For $0\leq i\leq r$, let $c_i\in\Z[X]/(p,\phi)$ be
$$
\displaystyle
c_i=
\left\lbrace \begin{array}{ll}
\displaystyle \overline{\left( \frac{a_i(x)}{p^{u_i}}\right) }\in\Z[X]/(p,\phi), & $if $(i,u_i)$ lies on $N,\\
&\\
0,& $otherwise$.
\end{array}\right\rbrace 
$$
Let $S$ be one of the sides of $N$. Let $\lambda=\frac{-h}{e}$ be the slope of $S$, where $h,e$ are positive coprime integers. Let $l$ be the length of the projection of $S$ onto the $x$-axis, and $d:=\frac{l}{e}$.\\
Let $t$ be the abscissa of the initial vertex of $S$, then the polynomial
$$R_\lambda(f)(Y)=c_t+c_{t+e}Y+\ldots+c_{t+(d-1)e}Y^{d-1}+c_{t+de}Y^d\in\left( \Z[X]/(p,\phi)\right)[Y]
$$
is called the {\bf residual polynomial} attached to $S$.\\
If $R_\lambda(f)(Y)$ is separable for each sides of the $\phi$-polygon, then $f$ is called \textbf{$\phi$-regular}.
\begin{theorem}\label{Monti}
(See \cite{GMN} Theorem 4.18) Let $f\in\Z[X]$, and choose monic polynomials $\phi_1,\ldots,\phi_k\in \Z[X]$, whose reductions modulo $p$ are the different irreducible factors of $\overline{f}$. Let $\alpha$ be a root of $f$, $K=\Q(\alpha)$, then 
$$\ind_p(\alpha)=v_p(({\cal O}_K^+:\Z[\alpha]^+))\geq \ind_{\phi_1}(f)+\ldots+\ind_{\phi_k}(f).
$$
Furthermore, equality holds if and only if $f$ is $\phi_i$-regular, for every $\phi_i$.\\
\end{theorem}

\section{An integral basis of $\Q(\sqrt[n]{m})$, for  $n=p^k$ }\label{s2}

In this section we investigate the integral bases of pure number fields of type $M=\Q(\sqrt[n]{m})$, where $m$ is square-free, and the degree $n$ of these fields is a power of a prime $p$.

\begin{lemma}\label{phatv}
Let $m\neq\pm1$ be a square-free integer, $p$ be a prime, $k$ be a positive integer, $\alpha$ be a root of $X^{p^k}-m$ and $M=\Q(\alpha)$. Let $r$ be the remainder of $m$ modulo $p^{k+1}$ and 
$$s:=v_p(m^p-m)-1.$$
For any $t\in\N$, let $h^{(r)}_t(X)\in\Z[X]$ be 
$$h^{(r)}_t(X):=X^{p^k-p^{k-t}}+rX^{p^k-2p^{k-t}}+\ldots+r^{p^t-2}X^{p^{k-t}}+r^{p^t-1}=\frac{X^{p^k}-r^{p^t}}{X^{p^{k-t}}-r}.$$
Then for any $0\leq t\leq \min\{s,k\}$, 
$$\frac{1}{p^t}h^{(r)}_t(\alpha)=\alpha^{p^k-p^{k-t}}+r\alpha^{p^k-2p^{k-t}}+\ldots+r^{p^t-2}\alpha^{p^{k-t}}+r^{p^t-1}=\frac{m-r^{p^t}}{\alpha^{p^{k-t}}-r}\in\Q(\alpha)$$
is an algebraic integer.
\end{lemma}

\begin{proof}
First we define a shift of the polynomial $h^{(r)}_t$. Let 
$$H_t(X):=X^{p^k-p^{k-t}}+mX^{p^k-2p^{k-t}}+\ldots+m^{p^t-2}X^{p^{k-t}}+m^{p^t-1}=\frac{X^{p^k}-m^{p^t}}{X^{p^{k-t}}-m}.$$
Since $p^{k+1}\mid m-r$, we have $H_t(X)-h^{(r)}_t(X)\in p^{k+1}\cdot\Z[X]$, and $t\leq k$ implies, that  
$$\frac{1}{p^t}h^{(r)}_t(\alpha)\in {\cal O}_M \quad\mbox{ if and only if }\quad \frac{1}{p^t}H_t(\alpha)\in {\cal O}_M.$$
Therefore, it is enough to show, that $\frac{1}{p^t}H_t(\alpha)$ is an algebraic integer.\\
Let $\beta_t:=\alpha^{p^{k-t}}$ and $\ell:=p^t$.
The minimal polynomial of $\beta_t$ is $X^\ell-m$, and 
$$H_t(\alpha)=\frac{m-m^\ell}{\beta_t-m}$$
is a rational transformation of $\beta_t$, so it is easy to determine the minimal polynomial of $H_t(\alpha)$.\\
Let
$$f(X):=\frac{(m-m^\ell+mX)^\ell-mX^\ell}{m-m^\ell}.$$
By substituting $H_t(\alpha)$ into $f$, we can see, that it is a root of $f$:
$$f(H_t(\alpha))=\frac{\left(m-m^{\ell}+m\left(\frac{m-m^\ell}{\beta_t-m}\right)\right)^\ell-m\left(\frac{m-m^\ell}{\beta_t-m}\right)^\ell}{m-m^\ell}=\frac{\left(\frac{\beta_t(m-m^\ell)}{\beta_t-m}\right)^\ell-m\left(\frac{m-m^\ell}{\beta_t-m}\right)^\ell}{m-m^\ell}=0.$$
Furthermore, $\Q(\beta_t)\subset\Q(H_t(\beta_t))$, so the degree of the minimal polynomial of $H_t(\beta_t)$ is at least $\ell$.
Consequently, the minimal polynomial of $H_t(\beta_t)$ is $f(X)$.\\
Let $a_i$ be the coefficient of $X^i$ in $f$. Since the minimal polynomial of $H_t(\beta_t)$ is $f(X)=\sum_{i=0}^{\ell}a_iX^i$, the minimal polynomial of $\frac{1}{\ell}H_t(\beta_t)$ is
$$\sum_{i=0}^{\ell}\frac{1}{\ell^{\ell-i}}a_iX^i.$$
In order to prove that $\frac{1}{\ell}H_t(\beta_t)$ is an algebraic integer, it is sufficient to show that $\ell^i\mid a_{\ell-i}$ for all $i\in \{0,\ldots,\ell\}$.
It is easy to see that $a_\ell=-1$ and for $i\in \{0,\ldots,\ell-1\}$,
$$a_i=\binom{\ell}{i}m^i(m-m^\ell)^{\ell-i-1}.$$
By using that for $i\in \{1,\ldots,\ell-1\}$, 
$$
v_p\left(\binom{\ell}{\ell-i}\right)=v_p(\ell)-v_p(i)
$$
and 
$$v_p(m-m^\ell)\geq v_p(m-m^p)>t,$$
it is straightforward to check, that  $\ell^i\mid a_{\ell-i}$ for $i=0,1,\ldots,\ell$, so $\frac{1}{p^t}H_t(\alpha)$ is indeed an algebraic integer.
\end{proof}

\begin{corollary}\label{skk}
We keep the notation of Lemma \ref{phatv}. If $s<k$, then the following elements are linearly independent algebraic integers over in $M$ over $\Q$:
$$\renewcommand{\arraystretch}{2.5}\left\lbrace \begin{array}{*3{>{\displaystyle}c}c>{\displaystyle}c}
\frac{h^{(r)}_0(\alpha)}{p^0},&
\alpha \cdot \frac{ h^{(r)}_0(\alpha)}{p^0},&  
\alpha^2 \cdot \frac{ h^{(r)}_0(\alpha)}{p^0},&  \ldots&,  
\alpha^{p^k-p^{k-1}-1} \cdot \frac{ h^{(r)}_0(\alpha)}{p^0},\\
\frac{h^{(r)}_1(\alpha)}{p^1},&  
\alpha \cdot \frac{ h^{(r)}_1(\alpha)}{p^1},&  
\alpha^2 \cdot \frac{ h^{(r)}_1(\alpha)}{p^1},& \ldots&,  
\alpha^{p^{k-1}-p^{k-2}-1} \cdot \frac{ h^{(r)}_1(\alpha)}{p^1},\\
\frac{h^{(r)}_2(\alpha)}{p^2},&  
\alpha \cdot \frac{ h^{(r)}_2(\alpha)}{p^2},&  
\alpha^2 \cdot \frac{ h^{(r)}_2(\alpha)}{p^2},& \ldots&,  
\alpha^{p^{k-2}-p^{k-3}-1} \cdot \frac{ h^{(r)}_2(\alpha)}{p^2},\\
&&&\vdots&\\
\frac{h^{(r)}_{s-1}(\alpha)}{p^{s-1}},&  
\alpha \cdot \frac{ h^{(r)}_{s-1}(\alpha)}{p^{s-1}},&  
\alpha^2 \cdot \frac{ h^{(r)}_{s-1}(\alpha)}{p^{s-1}},& \ldots&,  
\alpha^{p^{k-s+1}-p^{k-s}-1} \cdot \frac{ h^{(r)}_{s-1}(\alpha)}{p^{s-1}},\\
\frac{h^{(r)}_{s}(\alpha)}{p^{s}},&  
\alpha \cdot \frac{ h^{(r)}_{s}(\alpha)}{p^{s}},&  
\alpha^2 \cdot \frac{ h^{(r)}_{s}(\alpha)}{p^{s}},& \ldots&,  
\alpha^{p^{k-s}-1} \cdot \frac{ h^{(r)}_{s}(\alpha)}{p^{s}}.
\end{array}\right\rbrace $$
\end{corollary}

\begin{proof}
By Lemma \ref{phatv} for $t\in \{0,\ldots,\min\{s,k\}\}$, $h^{(r)}_t(\alpha)/p^t$
are algebraic integers, whence for any $j\in \N$,
$$\alpha^j\cdot \frac{h^{(r)}_t(\alpha)}{p^t}$$
are also algebraic integers. Furthermore, the degree of $h^{(r)}_t(X)$ is $p^{k}-p^{k-t}$, so the degrees of $\alpha$ in the elements of the set above are consecutively increasing by one, from 1 to $p^k-1$, which implies that these elements are linearly independent over $\Q$.
\end{proof}

\begin{corollary}\label{snk}
We keep the notation of Lemma \ref{phatv}. If $k\leq s$, then the following elements are linearly independent algebraic integers in $M$over $\Q$:
$$\renewcommand{\arraystretch}{2.5}\left\lbrace \begin{array}{*3{>{\displaystyle}c}c>{\displaystyle}c}
\frac{h^{(r)}_0(\alpha)}{p^0},&
\alpha \cdot \frac{ h^{(r)}_0(\alpha)}{p^0},&  
\alpha^2 \cdot \frac{ h^{(r)}_0(\alpha)}{p^0},&  \ldots&,  
\alpha^{p^k-p^{k-1}-1} \cdot \frac{ h^{(r)}_0(\alpha)}{p^0},\\
\frac{h^{(r)}_1(\alpha)}{p^1},&  
\alpha \cdot \frac{ h^{(r)}_1(\alpha)}{p^1},&  
\alpha^2 \cdot \frac{ h^{(r)}_1(\alpha)}{p^1},& \ldots&,  
\alpha^{p^{k-1}-p^{k-2}-1} \cdot \frac{ h^{(r)}_1(\alpha)}{p^1},\\
\frac{h^{(r)}_2(\alpha)}{p^2},&  
\alpha \cdot \frac{ h^{(r)}_2(\alpha)}{p^2},&  
\alpha^2 \cdot \frac{ h^{(r)}_2(\alpha)}{p^2},& \ldots&,  
\alpha^{p^{k-2}-p^{k-3}-1} \cdot \frac{ h^{(r)}_2(\alpha)}{p^2},\\
&&&\vdots&\\
\frac{h^{(r)}_{k-1}(\alpha)}{p^{k-1}},&  
\alpha \cdot \frac{ h^{(r)}_{k-1}(\alpha)}{p^{k-1}},&  
\alpha^2 \cdot \frac{ h^{(r)}_{k-1}(\alpha)}{p^{k-1}},& \ldots&,  
\alpha^{p^{1}-p^{0}-1} \cdot \frac{ h^{(r)}_{k-1}(\alpha)}{p^{k-1}},\\
&&\frac{h^{(r)}_{k}(\alpha)}{p^{k}}.&&
\end{array}\right\rbrace$$
\end{corollary}
\begin{proof}
Similarly to the previous corollary, these elements are algebraic integers, and the degrees of $\alpha$ in these elements are consecutively increasing by one, so they are linearly independent over $\Q$.
\end{proof}

\begin{remark}\label{disk}
In Corollary \ref{skk} and \ref{snk}, we gave a basis of $M$ over $\Q$ containing algebraic integers. The transition matrix from the basis $\left( 1,\alpha,\alpha^2,\ldots,\alpha^{p^k-1}\right) $ to this new basis is triangular, and having diagonal elements of the form $\frac{1}{p^i}$. 
\begin{itemize}
\item If $s<k$ then the determinant of the basis transformation matrix is 
$$\left( \prod_{t=0}^{s-1}\left( \frac{1}{p^t}\right) ^ {p^{k-t}-p^{k-t-1}}\right) \cdot \left( \frac{1}{p^s}\right) ^ {p^{k-s}}=\left( \frac{1}{p}\right) ^{\left( \frac{p^k-p^{k-s}}{p-1}\right)} .$$
\item If $k\leq s$ then the determinant of the basis transformation matrix is 
$$\left( \prod_{t=0}^{k-1}\left( \frac{1}{p^t}\right) ^ {p^{k-t}-p^{k-t-1}}\right) \cdot \left( \frac{1}{p^k}\right)=\left( \frac{1}{p}\right) ^{\left( \frac{p^k-1}{p-1}\right)} .$$
\end{itemize}
The discriminant of the new basis can be computed as the product of the discriminant of $\left( 1,\alpha,\alpha^2,\ldots,\alpha^{p^k-1}\right) $ and the square of the determinant of the basis transformation matrix. The discriminant $D_h$ of the new basis is
\begin{itemize}
\item if $s<k$ then 
$$D_h=\frac{D(\alpha)}{\left( p^{\left( \frac{p^k-p^{k-s}}{p-1}\right)}\right)^2 }, $$
\item if $k\leq s$ then 
$$D_h=\frac{D(\alpha)}{\left( p^{\left( \frac{p^k-1}{p-1}\right)}\right)^2 } .$$
\end{itemize}
\end{remark}

\begin{lemma}\label{pell}
Let $m$ be an integer, and $p\nmid m$ be a prime. Then for any $k\in\N$
$$v_p(m^p-m)=v_p(m^{p^k}-m).$$
\end{lemma}

\begin{proof}
It is easy to see by induction, that if $m^p\equiv m\mod{p^{t}}$, then $$m^{p^k}\equiv m^{p^{k-1}}\equiv\ldots\equiv m^{p^2}\equiv m^p\equiv m\mod{p^t},$$
so $v_p(m^{p^k}-m)\geq v_p(m^p-m)$.\\
On the other hand, if $m^{p^k}\equiv m\mod{p^t}$, then 
$$m^{p^k-1}\equiv 1\mod{p^t}$$
and by the Euler's theorem, $$m^{p^t-p^{t-1}}\equiv 1\mod{p^t}.$$
These two congruences imply that 
$$m^{\gcd(p^k-1,p^t-p^{t-1})}\equiv 1\mod{p^t}.$$
Since $\gcd(p^k-1,p^t-p^{t-1})=p-1$, we get that $m^p\equiv m\mod{p^t}$, and then $v_p(m^p-m)\geq v_p(m^{p^k}-m)$.\\
\end{proof}

\begin{theorem}\label{indp}
Let $m\neq\pm1$ be a square-free integer, $p$ be a prime, such that $p\nmid m$, and $k$ be a positive integer.
Let $f(X)=X^{p^k}-m$ and $\alpha$ be a root of $f$. Set $s:=v_p(m^p-m)-1$.
\begin{itemize}
\item If $s\leq k$, then 
$$\ind_p(\alpha)=\frac{p^k-p^{k-s}}{p-1}.$$
\item If $k<s$, then
$$\ind_p(\alpha)=\frac{p^k-1}{p-1}.$$
\end{itemize}
\end{theorem}

\begin{proof}
We intend to use Theorem \ref{Monti}. As the first step, we construct the principal Newton polygons of $f$. The factorization of $f$ modulo $p$ is 
$$\overline{f(X)}\equiv \overline{(X-m)}^{p^k}\mod{p},$$
which means that we have only one factor $\phi(X)=X-m$. Since $\phi(X)$ is of degree 1, the $\phi$-adic development of $f$ is in fact the Taylor series of $f$ at $m$:
$$f(X)=\sum_{i=0}^{p^k}\frac{f^{(i)}(m)}{i!}\phi(x)^i.$$
So the coefficients of the $\phi$-adic development of $f$ are:
$$a_0=m^{p^k}-m,$$
$$a_i=\frac{(X^{p^k}-m)^{(i)}(m)}{i!}=\binom{p^k}{i}m^{p^k-i},\quad i=1,\ldots,p^k.$$
The next step is to calculate the $p$-valuation of these coefficients. \\
For $i\in \{1,\ldots,{p^k}\}$, we recall that $m$ is coprime to $p$, so 
$$v_p(a_i)=v_p\left(\binom{p^k}{i}\right)=k-v_p(i).$$
Up to now, we have the following $p^k$ points on the plane:
$$\left\lbrace \left( i,k-v_p(i)\right):i\in \{1,\ldots,p^k\}\right\rbrace .$$
It is easy to check, that the extremal points of the lower convex hull of the points above are 
$$\left\lbrace P_j:=(p^j,k-j):j\in \{0,\ldots,k\}\right\rbrace .$$
For any $1\leq i\leq p^k$, the point $(i,k-v_p(i))$ is lying above the line segment $\overline{P_jP_{j+1}}$, where $p^j\leq i<p^{j+1}$. Furthermore, the slope of $\overline{P_jP_{j+1}}$ is $-(p^j(p-1))^{-1}$, which is increasing in $j$, so the graph created by the line segments $\left\lbrace \overline{P_jP_{j+1}},\;\;j=0,\ldots,k-1\right\rbrace $
is indeed convex, and there are no points below it. (See also T.A.Gassert \cite{maps}.)\\
We have one more point belonging to the coefficient $a_0$. The ordinate of this point is $v_p(a_0)=v_p(m^\ell-m)$, and by Lemma \ref{pell} it equals $v_p(m^p-m)=s+1$. So the initial point of the first line segment of the principal $\phi$-Newton polygon of $f$, is $P_N:=(0,s+1)$.
\begin{itemize}
\item Assume that $s<k$. If $p>2$, then the extremal points of the principal Newton polygon ordered by their abscissas are:
$$
\left\lbrace P_N,P_{k-s},P_{k-s+1},\ldots,P_k  \right\rbrace.
$$
We only have to verify that each points $P_0,\ldots,P_{k-s-1}$ are lying above the segment $\overline{P_NP_{k-s}}$ and the slope of this segment is less than the slope of the next side $\overline{P_{k-s}P_{k-s+1}}$.\\
The ordinates of the points $P_0,\ldots,P_{k-s-1}$ are greater or equal to $s+1$, and the ordinates of the points of the segment $\overline{P_NP_{k-s}}$ are between $s$ and $s+1$, so these point are indeed lying above it. The slope of $\overline{P_NP_{k-s}}$ is $(-(p^{k-s})^{-1})$, and the slope of $\overline{P_{k-s}P_{k-s+1}}$ is $(-(p^{k-s}(p-1))^{-1})$, which is greater than the previous one, since $p>2$.\\
One can see, that in this case, the degree of each residual polynomial is one, whence, they are separable.\\ 
If $p=2$, then the only difference is that the slopes of  $\overline{P_NP_{k-s}}$ and $\overline{P_{k-s}P_{k-s+1}}$ are equal, so the extremal points of the principal Newton polygon ordered by their abscissas are:
$$
\left\lbrace P_N,P_{k-s+1},P_{k-s+2},\ldots,P_k  \right\rbrace.
$$
In this case the residual polynomial attached to the first side is of degree 2, and it contains $P_N, P_{k-s}$ and $P_{k-s+1}$. Since $\phi$ is of degree one, then the polynomial ring $\left(\Z[X]/(p,\phi)\right)[Y]$ is isomorphic to $\F_p[Y]$. It is easy to determine, that the residual polynomial attached to the first side is $Y^2+Y+1$, which is separable. The residual polynomials attached to the other sides are of degree one, so they are also separable.
\item Assume that $k\leq s$. If $p>2$ or $k<s$, then the slope of $\overline{P_NP_0}$ is $k-(s+1)$, which is less than $-(p-1)^{-1}$, the slope of $\overline{P_0P_1}$, so in this case the extremal points are 
$$
\left\lbrace P_N,P_0,P_1,\ldots,P_k  \right\rbrace.
$$
Furthermore, the residual polynomials are separable, since they are of degree one.\\
If $p=2$ and $k=s$, then the slope of $\overline{P_NP_0}$ is equal the slope of $\overline{P_0P_1}$, so the extremal points of the principal Newton polygon ordered by their abscissas are
$$
\left\lbrace P_N,P_1,P_2,\ldots,P_k  \right\rbrace.
$$
In this case the residual polynomial attached to the first side is again $Y^2+Y+1$, and the others are of degree 1, so all of them are separable.
\end{itemize}
This consideration implies, that $f$ is always regular, hence we can apply Theorem \ref{Monti}, which gives
$$\ind_p(\alpha)=\ind_\phi(f).$$
The last step is to calculate the $\phi$-index of $f$. For this, we have to count the points lying strictly in the first quadrant, and below or on the principal $\phi$-Newton polygon of $f$. Each integral point in this area with ordinate $y$ can have abscissa $1\leq x\leq p^{k-y}$, which result in $p^{k-y}$ points.
\begin{itemize}
\item If $0\leq s < k$ then the ordinates of our points are less than $s+1$, so in this case there are 
$$\ind_p(\alpha)=\ind_\phi(f)=\sum_{j=1}^{s}p^{k-j}=\frac{p^k-p^{k-s}}{p-1}$$
points with integral coordinates.
\item If $k\leq s$ then the leftmost side does not exclude any integral points, so in this case there are 
$$\ind_p(\alpha)=\ind_\phi(f)=\sum_{j=1}^{k}p^{k-j}=\frac{p^k-1}{p-1}$$
points with integral coordinates.
\end{itemize}

\end{proof}
\begin{theorem}\label{fo}
Let $m\neq\pm1$ be a square-free integer, $k$ be a positive integer, $p$ be a prime, such that $p\nmid m$, $n=p^k$, and $s:=v_p(m^p-m)-1$. 
If $s<k$, then the elements listed in Corollary \ref{skk}, otherwise, the elements listed in Corollary \ref{snk} form an integral basis of $\Q(\sqrt[n]{m})$.
\end{theorem}

\begin{proof}
By Remark \ref{disk}, in Corollaries \ref{skk} and \ref{snk}, we gave a $\Q$-basis of $\Q(\sqrt[n]{m})$, such that its discriminant is equal to $p^{\ind_p(\sqrt[n]{m})}$, with $\ind_p(\sqrt[n]{m})$ given in Theorem \ref{indp}. By Theorem \ref{nes}, 
$$\ind(\sqrt[n]{m})=p^{\ind_p(\sqrt[n]{m})},$$
whence the discriminant of the basis in Corollary \ref{skk} or \ref{snk}, respectively, is equal to the discriminant of $\Q(\sqrt[n]{m})$, so it is an integral basis of $\Q(\sqrt[n]{m})$.
\end{proof}

\begin{theorem}\label{pmidm}
Let $m\neq\pm1$ be a square-free integer, $k$ be a positive integer, $p$ be a prime, such that $p\mid m$, and $n=p^k$. Then 
$$\left( 1,\sqrt[n]{m},\ldots,\sqrt[n]{m^{n-1}}\right) $$
is an integral basis of $\Q(\sqrt[n]{m})$.
\end{theorem}

\begin{proof}
By Theorem \ref{nes}, in this case $\ind(\sqrt[n]{m})=p^{\ind_p(\sqrt[n]{m})}$. Since $m$ is square-free, its minimal polynomial $x^n-m$ is $p$-Eiseinstein whence by Lemma \ref{ein}, $\ind_p(\sqrt[n]{m})=0$. This implies, that $\ind(\sqrt[n]{m})=p^{\ind_p(\sqrt[n]{m})}=1$, and therefore
$$\left( 1,\sqrt[n]{m},\ldots,\sqrt[n]{m^{n-1}}\right) $$
is an integral basis of $\Q(\sqrt[n]{m})$.
\end{proof}

\begin{theorem}\label{per}
If $k$ is a positive integer, $p$ is a prime, $n=p^k$ and $n_0=p^{k+1}$, and $m\neq\pm1$ is a square-free integer, then an integral basis of the pure fields of type $\Q(\sqrt[n]{m})$ is repeating periodically modulo $n_0$.
\end{theorem}

\begin{proof} 
First let $m$ be an integer, such that $p\nmid m$, and we keep the notation of Lemma \ref{phatv}.\\
If $s<k$, then consider the set
$$\renewcommand{\arraystretch}{2.5}\left\lbrace \begin{array}{*3{>{\displaystyle}c}c>{\displaystyle}c}
\frac{h^{(r)}_0(X)}{p^0},&
X \cdot \frac{ h^{(r)}_0(X)}{p^0},&  
X^2 \cdot \frac{ h^{(r)}_0(X)}{p^0},&  \ldots&,  
X^{p^k-p^{k-1}-1} \cdot \frac{ h^{(r)}_0(X)}{p^0},\\
\frac{h^{(r)}_1(X)}{p^1},&  
X \cdot \frac{ h^{(r)}_1(X)}{p^1},&  
X^2 \cdot \frac{ h^{(r)}_1(X)}{p^1},& \ldots&,  
X^{p^{k-1}-p^{k-2}-1} \cdot \frac{ h^{(r)}_1(X)}{p^1},\\
\frac{h^{(r)}_2(X)}{p^2},&  
X \cdot \frac{ h^{(r)}_2(X)}{p^2},&  
X^2 \cdot \frac{ h^{(r)}_2(X)}{p^2},& \ldots&,  
X^{p^{k-2}-p^{k-3}-1} \cdot \frac{ h^{(r)}_2(X)}{p^2},\\
&&&\vdots&\\
\frac{h^{(r)}_{s-1}(X)}{p^{s-1}},&  
X \cdot \frac{ h^{(r)}_{s-1}(X)}{p^{s-1}},&  
X^2 \cdot \frac{ h^{(r)}_{s-1}(X)}{p^{s-1}},& \ldots&,  
X^{p^{k-s+1}-p^{k-s}-1} \cdot \frac{ h^{(r)}_{s-1}(X)}{p^{s-1}},\\
\frac{h^{(r)}_{s}(X)}{p^{s}},&  
X \cdot \frac{ h^{(r)}_{s}(X)}{p^{s}},&  
X^2 \cdot \frac{ h^{(r)}_{s}(X)}{p^{s}},& \ldots&,  
X^{p^{k-s}-1} \cdot \frac{ h^{(r)}_{s}(X)}{p^{s}}.
\end{array}\right\rbrace $$
If $k\leq s$, then consider the set
$$\renewcommand{\arraystretch}{2.5}\left\lbrace \begin{array}{*3{>{\displaystyle}c}c>{\displaystyle}c}
\frac{h^{(r)}_0(X)}{p^0},&
X \cdot \frac{ h^{(r)}_0(X)}{p^0},&  
X^2 \cdot \frac{ h^{(r)}_0(X)}{p^0},&  \ldots&,  
X^{p^k-p^{k-1}-1} \cdot \frac{ h^{(r)}_0(X)}{p^0},\\
\frac{h^{(r)}_1(X)}{p^1},&  
X \cdot \frac{ h^{(r)}_1(X)}{p^1},&  
X^2 \cdot \frac{ h^{(r)}_1(X)}{p^1},& \ldots&,  
X^{p^{k-1}-p^{k-2}-1} \cdot \frac{ h^{(r)}_1(X)}{p^1},\\
\frac{h^{(r)}_2(X)}{p^2},&  
X \cdot \frac{ h^{(r)}_2(X)}{p^2},&  
X^2 \cdot \frac{ h^{(r)}_2(X)}{p^2},& \ldots&,  
X^{p^{k-2}-p^{k-3}-1} \cdot \frac{ h^{(r)}_2(X)}{p^2},\\
&&&\vdots&\\
\frac{h^{(r)}_{k-1}(X)}{p^{k-1}},&  
X \cdot \frac{ h^{(r)}_{k-1}(X)}{p^{k-1}},&  
X^2 \cdot \frac{ h^{(r)}_{k-1}(X)}{p^{k-1}},& \ldots&,  
X^{p^{1}-p^{0}-1} \cdot \frac{ h^{(r)}_{k-1}(X)}{p^{k-1}},\\
&&\frac{h^{(r)}_{k}(X)}{p^{k}}.&&
\end{array}\right\rbrace $$
By Theorem \ref{fo}, if we substitute $\sqrt[n]{m}$ in place of $X$ in the appropriate set above, then we obtain an integral basis of $\Q(\sqrt[n]{m})$. It is easy to see, that $\min\{ v_p(m^p-m)-1,k\}$ depends only on the remainder of $m$ modulo $n_0$ whence the coefficients and the number of the polynomials in these sets depend also only on the remainder of $m$ modulo $n_0$.\\
If $p\mid m$, then by Theorem \ref{pmidm}, if we substitute $\sqrt[n]{m}$ in place of $X$ in 
$$(1,X,X^2,\ldots,X^{n-1}),$$
then we obtain an integral basis of $\Q(\sqrt[n]{m})$. However, $p$ divides $m$ if and only if $p$ divides the remainder of $m$ modulo $n_0$, so this set also depends only on the remainder of $m$ modulo $n_0$. This means by definition, that if $m$ is square-free, then an integral basis of $\Q(\sqrt[n]{m})$ is repeating periodically modulo $n_0$.\\
\end{proof}

\subsection{Example}\label{ex1}
As an example, we determine an integral basis of the infinite parametric family of pure number fields of type $\Q(\sqrt[9]{27t+1})$.
\begin{proposition}\label{9}
Let $t$ be a nonzero integer, and $\alpha$ be a root of $X^{9}-(27t+1)$. If $27t+1$ is square-free, then by substituting $\alpha$ in place of $X$ in 
$$\left(
1,X,X^2,X^3,X^4,X^5,
\frac{1+X^3+X^6}{3},\frac{X+X^4+X^7}{3},
\frac{1+X+X^2+X^3+X^4+X^5+X^6+X^7+X^8}{9}
\right),
$$
then we obtain an integral basis of $\Q(\alpha)$.
\end{proposition}

\begin{proof}
Since $v_3((27t+1)^3-(27t+1))=v_3(19683t^3+2187t^2+54t)\geq3$, then by Theorem \ref{fo}, 
$$
\left(
h^{(1)}_0(\alpha),\alpha\cdot h^{(1)}_0(\alpha),\alpha^2\cdot h^{(1)}_0(\alpha),\alpha^3\cdot h^{(1)}_0(\alpha),\alpha^4\cdot h^{(1)}_0(\alpha),\alpha^5\cdot h^{(1)}_0(\alpha),
\frac{h^{(1)}_1(\alpha^3)}{3},\frac{\alpha\cdot h^{(1)}_1(\alpha^3)}{3},
\frac{h^{(1)}_2(\alpha)}{9}
\right)
$$
is an integral basis of $\Q(\alpha)$, where
$$h^{(1)}_0(X)=1,$$
$$h^{(1)}_1(X)=X^2+X+1,$$
$$h^{(1)}_2(X)=X^8+X^7+X^6+X^5+X^4+X^3+X^2+X+1.$$

It can also be proved using Theorem \ref{per}, without calculating the $h^{(1)}_i$ polynomials. By this corollary, if $m$ is square-free, then an integral basis of the fields $\Q(\sqrt[9]{m})$ is repeating periodically modulo 27. Now calculate with any computer algebra system an integral basis of $\Q\left( \sqrt[9]{-26}\right) $. We chose $-26$, since it is square-free and has remainder 1 modulo 27. Afterwards, represent the elements of this integral basis in the $\Q$-basis $\left( 1,\sqrt[9]{-26},\ldots,\sqrt[9]{(-26)^8}\right) $. (Most of the computer algebra systems calculate an integral basis of this form, usually based on the Round 4 algorithm). \\
Finally for any square-free values of $m$ having remainder 1 modulo 27, we obtain an integral basis of $\Q(\sqrt[9]{m})$, by substituting $\sqrt[9]{m}$ in place of $\sqrt[9]{-26}$ in this basis.
\end{proof}

\newpage
\section{An integral basis of $\Q(\sqrt[n]{m})$, where $n$ is composite}\label{s3}
\begin{lemma}\label{comp}
Let $m\neq\pm1$ be a square-free integer, $2\leq n_1,n_2$ be coprime integers and $n=n_1\cdot n_2$. Then 
$$
\ind(\sqrt[n]{m})=\left( \ind(\sqrt[n_1]{m})\right) ^{n_2}\cdot \left( \ind(\sqrt[n_2]{m})\right) ^{n_1}.
$$
\end{lemma}

\begin{proof}
Let
$$M=\Q(\sqrt[n]{m}),\qquad M_1=\Q(\sqrt[n_1]{m}), \qquad M_2=\Q(\sqrt[n_2]{m}),$$
and let $D_M$, $D_{M_1}$ and $D_{M_2}$ denote their discriminants respectively.\\
By Corollary \ref{den}, $\ind(\sqrt[n]{m})$, $\ind(\sqrt[n_1]{m})$ and $\ind(\sqrt[n_2]{m})$ are coprime to $m$. \\
It is well known (\cite{Nark} Chapter IV., Proposition 4.15), that for any tower of fields $M/L/\Q$:
$$D_M=N_{L/\Q}(D_{M/L})\cdot D_L^{[M:L]}.$$
Applying this with $L=M_1$ we get that $\left( D_{M_1}\right) ^{n_2}\mid D_M $, and with $L=M_2$ we get that $ \left( D_{M_2}\right) ^{n_1}\mid D_M$. So
$$
\lcm\left( \left( D_{M_1}\right) ^{n_2},\left( D_{M_2}\right) ^{n_1}\right) \mid D_M.
$$
Using that
$$D_{M}=\frac{D(\sqrt[n]{m})}{\ind\left( \sqrt[n]{m}\right)^2 }=\frac{(-1)^{\frac{n(n-1)}{2}}\cdot n^n\cdot m^{n-1}}{\ind\left( \sqrt[n]{m}\right)^2 },$$
$$
D_{M_1}=\frac{D(\sqrt[n_1]{m})}{\ind\left( \sqrt[n_1]{m}\right)^2 }=\frac{(-1)^{\frac{n_1(n_1-1)}{2}}\cdot n_1^{n_1}\cdot m^{n_1-1}}{\ind\left( \sqrt[n_1]{m}\right)^2 },$$
$$D_{M_2}=\frac{D(\sqrt[n_2]{m})}{\ind\left( \sqrt[n_2]{m}\right)^2 }=\frac{(-1)^{\frac{n_2(n_2-1)}{2}}\cdot n_2^{n_2}\cdot m^{n_2-1}}{\ind\left( \sqrt[n_2]{m}\right)^2 }.
$$
together with $\gcd(n_1,n_2)=1$, and $\gcd\left( \ind\left( \sqrt[n]{m}\right)\cdot\ind\left( \sqrt[n_1]{m}\right)\cdot\ind\left( \sqrt[n_2]{m}\right),m\right)=1$, we obtain
$$
\left. \lcm\left( \left( \frac{n_1^{n_1}}{\ind\left( \sqrt[n_1]{m}\right)^2 }\right) ^{n_2},\left( \frac{n_2^{n_2}}{\ind\left( \sqrt[n_2]{m}\right)^2 }\right) ^{n_1}\right) = \frac{n^n}{\left( \ind\left( \sqrt[n_1]{m}\right)^2 \right) ^{n_2}\cdot \left( \ind\left( \sqrt[n_2]{m}\right)^2 \right) ^{n_1} }\;\;\right|\;\; \frac{n^n}{\ind\left( \sqrt[n]{m}\right)^2 }.
$$
On the other hand, consider the composite basis of the integral bases of the fields $M_1$ and $M_2$. Let $\{1,\psi_2,\ldots,\psi_{n_1}\}$ be an integral basis of $M_1$ and 
$\{1,\omega_2,\ldots,\omega_{n_2}\}$ be an integral basis of $M_2$. Then 
$$\{1,\psi_2,\ldots,\psi_{n_1},\omega_2,\psi_2\omega_2,\ldots,\psi_{n_1}\omega_2,\ldots\ldots,\omega_{n_2},\psi_2\omega_{n_2},\ldots,\psi_{n_1}\omega_{n_2}\}$$
is a $\Q$-basis for $M$, because these $n_1\cdot n_2=n$ numbers are independent over $\Q$ due to $\gcd(n_1,n_2)=1$.\\
The discriminant of this basis containing algebraic integers is $D_{M_1}^{n_2}\cdot D_{M_2}^{n_1}$. Since $D_M$ divides the discriminant of any $\Q$-basis of $K$ containing algebraic integers, we have
$$D_M\mid D_{M_1}^{n_2}\cdot D_{M_2}^{n_1}.$$
Using the coprimality of the indices to $m$ again, this implies:
$$
\left. \frac{n^n}{\ind\left( \sqrt[n]{m}\right)^2 } \;\;\right|\;\;  \frac{n^n}{\left( \ind\left( \sqrt[n_1]{m}\right)^2 \right) ^{n_2}\cdot \left( \ind\left( \sqrt[n_2]{m}\right)^2 \right) ^{n_1} }.
$$
So we have:
$$\frac{n^n}{\ind\left( \sqrt[n]{m}\right)^2 }=\frac{n^n}{\left( \ind\left( \sqrt[n_1]{m}\right)^2 \right) ^{n_2}\cdot \left( \ind\left( \sqrt[n_2]{m}\right)^2 \right) ^{n_1} },$$
which implies that 
$$
\ind(\sqrt[n]{m})=\left( \ind(\sqrt[n_1]{m})\right) ^{n_2}\cdot \left( \ind(\sqrt[n_2]{m})\right) ^{n_1}.
$$
\end{proof}

Using this Lemma, we can construct an integral basis of $\Q(\sqrt[n]{m})$.\\
Let $m\neq\pm1$ be a square-free integer, $2\leq n_1,n_2$ be coprime integers, $n=n_1\cdot n_2$, 
$$\alpha=\sqrt[n]{m},\qquad\alpha_1=\sqrt[n_1]{m},\qquad\alpha_2=\sqrt[n_2]{m},$$
$$M=\Q(\alpha),\qquad M_1=\Q(\alpha_1),\qquad M_2=\Q(\alpha_2).$$
Let $\psi_0\equiv1,\psi_1,\ldots,\psi_{n_1-1}\in\Q[X]$ and $\omega_0\equiv1,\omega_1,\ldots,\omega_{n_2-1}\in\Q[X]$ be polynomials such that $\deg(\psi_i)=i$, $\deg(\omega_i)=i$, and
$$
(\psi_0(\alpha_1),\psi_1(\alpha_1),\ldots,\psi_{n_1-1}(\alpha_1))
$$ 
is an integral basis of $M_1$, and
$$
(\omega_0(\alpha_2),\omega_1(\alpha_2),\ldots,\omega_{n_2-1}(\alpha_2))
$$
is an integral basis of $M_2$. \\
Extend these polynomials in the following way:
$$
\Psi_{i,j}(X):=X^j\cdot \psi_i(X^{n_2}),\qquad \left( i=0,\ldots, n_1-1, \;\; j=0,\ldots,n_2-1\right) ,
$$
$$
\Omega_{i,j}(X):=X^j\cdot \omega_i(X^{n_1}),\qquad \left( i=0,\ldots, n_2-1, \;\; j=0,\ldots,n_1-1\right) .
$$
For any $k\in\{0,\ldots,n-1\}$ there exist uniquely determined $a,d\in \{0,\ldots, n_1-1\}$ and $b,c\in \{0,\ldots, n_2-1\}$, such that 
$$\deg(\Psi_{a,b})=\deg(\Omega_{c,d})=k.$$
Let $u$ be the denominator of $\Psi_{a,b}$, $v$ be the denominator of $\Omega_{c,d}$.\\
(The denominator $c$ of a rational polynomial $p(X)\in\Q[X]$ is the smallest positive integer, such that $c\cdot p(X)\in\Z[X]$.)\\
For $k\in\{0,\ldots,n-1\}$ let $\gamma_k\in\Q[X]$ be a polynomial of degree $k$, such that
$${v\cdot \gamma_k-\Psi_{a,b}}\in\Z[X],$$
$${u\cdot \gamma_k-\Omega_{c,d}}\in \Z[X].$$

\begin{theorem}\label{n1n2}
With the notation above, $(\gamma_0(\alpha),\gamma_1(\alpha),\ldots,\gamma_{n-1}(\alpha))$ is an integral basis of $M$.
\end{theorem}

\begin{proof}
We will use induction by the number of prime divisors of $n$. First let $n_1=p_1^{k_1}$ and $n_2=p_2^{k_2}$. 
By Corollary \ref{per} we can choose polynomials $\psi_i$, $\omega_j\in \Q[X]$, such that $\deg(\psi_i)=i$, $\deg(\omega_j)=j$, and
$$
(\psi_0(\alpha_1),\psi_1(\alpha_1),\ldots,\psi_{n_1-1}(\alpha_1))
$$ 
is an integral basis of $M_1$, and
$$
(\omega_0(\alpha_2),\omega_1(\alpha_2),\ldots,\omega_{n_2-1}(\alpha_2))
$$
is an integral basis of $M_2$. \\
By definition, we have:
$$\deg(\Psi_{i,j})=j+i\cdot n_2,\qquad
\left( i=0,\ldots n_1-1, \;\; j=0,\ldots,n_2-1\right),$$
$$\deg(\Omega_{i,j})=j+i\cdot n_1, \qquad \left( i=0,\ldots n_2-1, \;\; j=0,\ldots,n_1-1\right).$$
So the degrees of $\Psi_{i,j}$-s, and similarly, the degrees of $\Omega_{i,j}$-s are distinct numbers from the set $\{0,1,\ldots,n-1\}$ whence
$$
\Psi_{0,0}(\alpha)=1,\Psi_{0,1}(\alpha),\Psi_{0,2}(\alpha),\ldots,\Psi_{n_1-1,n_2-1}(\alpha)\in\Q(\alpha)
$$ 
are linearly independent elements over $\Q$, so they form a $\Q$-basis for $M$. Furthermore, by the construction of the $\Psi_{i,j}$-s, these elements are algebraic integers, which implies, that $D_M$ divides the discriminant of this basis.\\
From the definition of $\Psi_{i,j}$, it is very easy to determine the discriminant $D_\Psi$ of the basis
$$
(\Psi_{0,0}(\alpha)=1,\Psi_{0,1}(\alpha),\Psi_{0,2}(\alpha),\ldots,\Psi_{n_1-1,n_2-1}(\alpha)),
$$
because the product of the denominators of the elements of this basis written in the $\Q$-basis $(1,\alpha,\ldots,\alpha^{n-1})$ is equal to the $n_2$-th power of the product of the denominators of $\psi_i$-s, which is exactly the index of $\alpha_1$. So we have
$$D_\Psi=D(\Psi_{0,0}(\alpha),\Psi_{0,1}(\alpha),\Psi_{0,2}(\alpha),\ldots,\Psi_{n_1-1,n_2-1}(\alpha))=\frac{D(\alpha)}{\left( \ind\left( \alpha_1\right)^2 \right) ^{n_2}}.$$
We have the same with the $\Omega_{i,j}$-s:
$$D_\Omega=\frac{D(\alpha)}{\left( \ind\left( \alpha_2\right)^2 \right) ^{n_1}}.
$$
By using the Chinese Remainder Theorem it is easy to find such polynomials $\gamma_k$, for which 
$${v\cdot \gamma_k-\Psi_{a,b}}\in\Z[X],$$
$${u\cdot \gamma_k-\Omega_{c,d}}\in \Z[X].$$
Since $u$ is the denominator of $\Psi_{a,b}$, it divides $\ind(\alpha_1)$, and similarly, $v$ divides $\ind(\alpha_2)$. These indices are coprime, so there exist $s,t\in\Z$, such that
$$su+tv=1,$$
whence
$s(v\cdot \gamma_k-\Psi_{a,b})+t({u\cdot \gamma_k-\Omega_{c,d}})=\gamma_k-s\cdot\Psi_{a,b}-t\cdot\Omega_{c,d}\in\Z[X]$.\\
This implies, that the $\gamma_k(\alpha)$-s are algebraic integers, furthermore, by their definitions, we can write every elements $\Psi_{i,j}(\alpha)$, and $\Omega_{i,j}(\alpha)$  in the basis $(1,\gamma_1(\alpha),\ldots,\gamma_{n-1}(\alpha))$ with integer coefficients. This means, that the discriminant $D_\gamma$ of $(1,\gamma_1(\alpha),\ldots,\gamma_{n-1}(\alpha))$ divides $D_\Psi$ and $D_\Omega$ and therefore
$$D_\gamma\mid \gcd(D_\Psi,D_\Omega).$$
Since $\gcd(\ind(\alpha_1),\ind(\alpha_2))=1$, by Lemma \ref{comp} we get
$$\gcd(D_\Psi,D_\Omega)=\frac{D(\alpha)}{\left( \ind\left( \alpha_1\right)^2 \right) ^{n_2}\cdot \left( \ind\left( \alpha_2\right)^2 \right) ^{n_1} }=\frac{D(\alpha)}{\ind\left( \alpha\right)^2 }=D_M,
$$
whence the discriminant of the basis $(1,\gamma_1(\alpha),\ldots,\gamma_{n-1}(\alpha))$ containing algebraic integers, divides the discriminant of $M$, which implies that it is an integral basis of $M$.\\
Notice, that by definition $\deg(\gamma_k)=k$, so we can repeat this process with $n_1=p_1^{k_1}p_2^{k_2}$ and $n_2=p_3^{k_3}$ to prove the theorem for any $n$ having at most three distinct prime divisor. Then by induction the statement of the theorem comes true for any $n\in \N$. 
\end{proof}

\begin{theorem}\label{per2}
Let $n\geq2$ with prime decomposition 
$$n=p_1^{k_1}\cdot p_2^{k_2}\cdot\ldots\cdot p_j^{k_j}$$
and 
$$n_0:=p_1^{k_1+1}\cdot p_2^{k_2+1}\cdot\ldots\cdot p_j^{k_j+1}.$$
If $m\neq\pm1$ is a square-free integer, then an integral basis of the pure fields of type $\Q(\sqrt[n]{m})$ is repeating periodically modulo $n_0$.
\end{theorem}

\begin{proof} 
For any $i\in \{1,\ldots,j\}$, an integral basis of the fields $\Q(\sqrt[p^{k_i}]{m})$ is repeating periodically modulo $p^{k_i+1}$, and by following the proof of Theorem \ref{n1n2} we can construct an integral basis for $\Q(\sqrt[n]{m})$ by using the integral bases of these subfields. So an integral basis of $\Q(\sqrt[n]{m})$ is repeating periodically modulo $n_0$.
\end{proof}

\subsection{Example}\label{ex2}
As an example we give a similar characterisation of the ring of integers of pure fields of degree $12$, as it was given in I.Gaál and L.Remete \cite{ibm} for degree $2\leq n\leq 10$.\\
One can use Theorem \ref{fo} and Theorem \ref{n1n2} to calculate the appropriate polynomials, but it is much more convenient to use Theorem \ref{per2}. By this theorem, an integral basis of $\Q(\sqrt[12]{m})$ is repeating periodically modulo $72$. So in order to determine an integral basis of any $\Q(\sqrt[12]{m})$, where $m$ is square-free, we have to follow the steps below.
\begin{itemize}
\item For any $r\in \{0,1,\ldots,71\}$, for which $\gcd(r,72)$ is square-free, chose an appropriate $t_r\in\Z$, such that $72t_r+r$ is square-free.
\item Calculate an integral basis of $\Q(\sqrt[12]{72t_r+r})$, in which the elements are written in the basis $(1,\sqrt[12]{72t_r+r},\ldots,\sqrt[12]{(72t_r+r)^{11}}).$
\item Substitute $\sqrt[12]{m}$ in place of $\sqrt[12]{72t_r+r}$ in these bases.
\end{itemize}
We obtained the following results by using MAPLE.

\begin{proposition}
Let $m\neq\pm1$ be a square-free integer. 
\begin{itemize}
\item if $m=72t+r$, where 
$r\in\{ 2,3,6,7,11,14,15,22,23,30,31,34,38,39,42,43,47,50,51,$\\$58,59,66,67,70\}$, then let\\
$\displaystyle B_r=(1, X, X^2, X^3, X^4, X^5, X^6, X^7, X^8, X^9, X^{10}, X^{11}),$
\item if $m=72t+r$, where $r\in\{ 10,19,46,55\}$, then let\\
{\footnotesize$\displaystyle  B_r=\left(1, X, X^2, X^3, X^4, X^5, X^6, X^7, \frac{X^8+X^4+1}{3}, \frac{X^9+X^5+X}{3},\frac{X^{10}+X^6+X^2}{3}, \frac{X^{11}+X^7+X^3}{3}\right),$}
\item if $m=72t+r$, where $r\in\{ 26,35,62,71\}$, then let\\
{\footnotesize$\displaystyle B_r=\left(1, X, X^2, X^3, X^4, X^5, X^6, X^7, \frac{X^8+2X^4+1}{3},\frac{ X^9+2X^5+X}{3},\frac{X^{10}+2X^6+X^2}{3},\frac{ X^{11}+2X^7+X^3}{3}\right),$}
\item if $m=72t+r$, where $r\in\{ 5,13,21,29,61,69\}$, then let\\
{\footnotesize$\displaystyle B_r=\left(1, X, X^2, X^3, X^4, X^5, \frac{X^6+1}{2}, \frac{X^7+X}{2}, \frac{X^8+X^2}{2}, \frac{X^9+X^3}{2}, \frac{X^{10}+X^4}{2}, \frac{X^{11}+X^5}{2}\right),$}
\item if $m=72t+r$, where $r\in\{25,3,41,49,57,65\}$, then let\\
{\footnotesize$\displaystyle B_r=\left(1, X, X^2, X^3, X^4, X^5, \frac{X^6+1}{2}, \frac{X^7+X}{2}, \frac{X^8+X^2}{2},  \frac{X^9+X^6+X^3+1}{4}, \frac{X^{10}+X^7+X^4+X}{4},\right.$\\
\begin{center}
$\displaystyle\left. \frac{X^{11}+X^8+X^5+X^2}{4}\right),$
\end{center}}
\item if $m=72t+r$, where $r=53$, then let\\
{\footnotesize$\displaystyle B_r=\left(1, X, X^2, X^3, X^4, X^5, \frac{X^6+1}{2}, \frac{X^7+X}{2}, \frac{X^8+2X^4+3X^2+4}{6},  \frac{X^9+2X^5+3X^3+4X}{6},\right.$\\
\begin{center}
$\displaystyle\left. \frac{X^{10}+2X^6+3X^4+4X^2}{6}, \frac{X^{11}+2X^7+3X^5+4X^3}{6}\right),$
\end{center}}
\item if $m=72t+r$, where $r=17$, then let\\
{\footnotesize$\displaystyle B_r=\left(1, X, X^2, X^3, X^4, X^5, \frac{X^6+1}{2}, \frac{X^7+X}{2}, \frac{X^8+2X^4+3X^2+4}{6},\frac{X^9+3X^6+8X^5+9X^3+4X+3}{12},\right.$\\
\begin{center}
$\displaystyle\left.  \frac{X^{10}+3X^7+2X^6+9X^4+4X^2+3X+6}{12}, \frac{X^{11}+X^8+2X^7+9X^5+8X^4+4X^3+9X^2+6X+4}{12}\right),$
\end{center}}
\item if $m=72t+r$, where $r=37$, then let\\
{\footnotesize$\displaystyle B_r=\left(1, X, X^2, X^3, X^4, X^5, \frac{X^6+1}{2}, \frac{X^7+X}{2}, \frac{X^8+4X^4+3X^2+4}{6}, \frac{X^9+4X^5+3X^3+4X}{6},\right.$\\
\begin{center}
$\displaystyle\left.  \frac{X^{10}+X^6+3X^4+4X^2+3}{6}, \frac{X^{11}+X^7+3X^5+4X^3+3X}{6}\right),$
\end{center}}
\item if $m=72t+r$, where $r=1$, then let\\
{\footnotesize$\displaystyle B_r=\left(1, X, X^2, X^3, X^4, X^5, \frac{X^6+1}{2}, \frac{X^7+X}{2}, \frac{X^8+4X^4+3X^2+4}{6},\frac{X^9+3X^6+4X^5+9X^3+4X+3}{12},  \right.$\\
 \begin{center}
 $\displaystyle \left. \frac{X^{10}+3X^7+4X^6+9X^4+4X^2+3X}{12},\frac{X^{11}+X^8+4X^7+9X^5+4X^4+4X^3+9X^2+4}{12}\right).$
 \end{center}}
If we substitute $\sqrt[12]{m}$ in place of $X$ in $B_r$, respectively, then we obtain an integral basis of $\Q(\sqrt[12]{m})$
\end{itemize}

\end{proposition}


\begin{thebibliography}{10}

\normalsize
\baselineskip=17pt

\bibitem{ber}
W.E.H.Berwick,
{\em Integral bases},
Cambridge University Press (Cambridge Tracts in Mathematics and Mathematical Physics No. 22) (1927).

\bibitem{ded}
R.Dedekind,
{\em Ueber die Anzahl der Idealklassen in reinen kubischen Zahlkörpern}, 
J. Reine Angew. Math. {\bf 121} (1900), 40--123.

\bibitem{fun}
T.Funakura,
{\em On integral bases of pure quartic fields},
Math. J. Okayama Univ. {\bf 26} (1984), 27--41.

\bibitem{FMN}
{L.{El Fadil}, J.{Montes} and E.{Nart}},
{\em {Newton polygons and $p$-integral bases of quartic number fields}}, {{J. Algebra Appl.}}, {\bf 11 (4)} (2012), 33p.

\bibitem{ibm}
{I.{Ga\'al} and L.{Remete}}, {\em Integral bases and monogenity of pure fields},
Journal of Number Theory {\bf 173} (2017), 129--146.

\bibitem{sextic}
{I.{Ga\'al} and L.{Remete}}, {\em Integral bases and monogenity of the simplest sextic fields}, Acta Arithmetica {\bf 182 (2)} (2018), 173--183.

\bibitem{maps}
T.A.Gassert, {\em A note on the monogenity of power maps}, Albanian Journal of Mathematics {\bf 11} (2017), 3--12.

\bibitem{GMN} 
J.Gu\`ardia, J.Montes and E.Nart, {\em Newton polygons of higher order in algebraic number theory,}
Trans. Amer. Math. Soc. \textbf{364 (1)} (2012) 361--416.


\bibitem{n3}
A.Hameed and T.Nakahara,
{\em Integral bases and relative monogenity of pure octic fields},
Bull. Math. Soc. Sci. Math. R\'epub. Soc. Roum., Nouv. S\'er. 
{\bf 58 (106)}, (2015) No.4, 419--433.

\bibitem{nnn}
A.Hameed, T.Nakahara, S.M.Husnine and S.Ahmad, 
{\em On existence of canonical number system in certain classes of pure algebraic number fields}, 
J. Prime Research in Mathematics, {\bf 7} (2011), 19--24.

\bibitem{Nark}
{W.{Narkiewicz}}, {\em {Elementary and analytic theory of algebraic numbers. 3rd ed.}}, 
{Springer Monogr. Math.} (2004).

\bibitem{oku}
K.Okutsu,
{\em Integral basis of the field $\Q(\sqrt[n]{a})$}, Proc. Japan Acad., Ser. A, {\bf 58} (1982), 219--222.

\bibitem{ore}
\O. Ore, \textit{Newtonsche Polygone in der Theorie der algebraischen Körper,} Math. Ann. \textbf{99} (1928)
84--117.

\bibitem{pohst}
M. {Pohst} and H. {Zassenhaus},
{\em Algorithmic algebraic number theory}, {Cambridge University Press,} (1989).

\end{thebibliography}
\end{document}